\documentclass[journal]{IEEEtran}

\usepackage{pmat}

\usepackage{float}

\usepackage{algorithmic}
\usepackage{algorithm}
\usepackage[table]{xcolor}
\usepackage{multirow}
\usepackage{amssymb,amsmath,amsthm,bm,amsfonts}
\usepackage{cases}
\usepackage{color}
\usepackage{lineno,hyperref}
\hypersetup{hidelinks} 
\usepackage{cite}               
\usepackage{graphicx}
\usepackage{subfigure}
\usepackage{doi}

\newtheorem{theorem}{Theorem}
\newtheorem{lemma}{Lemma}

\newtheorem{assumption}{Assumption}
\newtheorem{remark}{Remark}
\newtheorem{example}{Example}

\usepackage{stfloats}
\fnbelowfloat

\begin{document}
%
\title{Distributed Consensus of Heterogeneous Multi-Agent Systems Based on Feedforward Control}

%
%
\author{
\thanks{This work was supported by the National Natural Science Foundation of China (62103240),
the Original Exploratory Program Project of National Natural Science Foundation of China (62250056), the Joint Funds of the National Natural Science Foundation of China (U23A20325), the Major Basic Research of Natural Science Foundation of Shandong Province (ZR2021ZD14), the Youth Foundation of Natural Science Foundation of Shandong Province (ZR2021QF147) and the High-level Talent Team Project of Qingdao West Coast New Area (RCTD-JC-2019-05).}

\author{Liping Zhang, 
       Huanshui Zhang
\thanks{This work was supported by
the Original Exploratory Program Project of National Natural Science Foundation of China (62250056),
the National Natural Science Foundation of China (62103240),the Joint Funds of the National Natural Science Foundation of China (U23A20325), the Major Basic Research of Natural Science Foundation of Shandong Province (ZR2021ZD14), the Youth Foundation of Natural Science Foundation of Shandong Province (ZR2021QF147) and the High-level Talent Team Project of Qingdao West Coast New Area (RCTD-JC-2019-05).}
\thanks{L. Zhang and H. Zhang are with the College of Electrical Engineering and Automation,
Shandong University of Science and Technology, Qingdao 266590, China (e-mail:lpzhang1020@sdust.edu.cn;hszhang@sdu.edu.cn).}
}

}

%
%


\maketitle

\begin{abstract}
This paper studies the consensus problem of heterogeneous multi-agent systems by linear quadratic (LQ) optimal control theory.
Different from the existing consensus control  for heterogeneous systems where a reference generator requires to be designed,
which leads to a more complex design than that of the homogeneous case.
This paper develops a distributed asymptotically optimal consensus controller by designing a novel kind of observer and feedforward controller.
The proposed  controller is obtained by solving Riccati equations, and
it is a unified design framework  to address the consensus problem for heterogeneous and homogeneous systems.
Moreover, the convergence speed under the proposed distributed asymptotically optimal controller is proved to be faster than those of traditional methods.
Simulation example  further verifies the effectiveness of the proposed scheme and a much faster convergence speed than the existing algorithms.
\end{abstract}


%
\IEEEpeerreviewmaketitle

\section{Introduction}
In recent years, multi-agent systems has attracted considerable attentions for its
extensive engineering applications in unmanned aerial vehicle and satellite formation,
multi-robot systems, and wireless sensor networks \cite{Ren2007a,Olfati-Saber2007,Yang2022}.
 The consensus problem of multi-agent systems is an important and fundamental problem,
 whose essential task is to design a distributed controller based on the interaction with neighbor agents such that all agents achieve an agreement for certain variable of common interest.
A general framework of the consensus problem for first-order integrator networks is addressed with a fixed topology in \cite{Olfati-Saber2004}.
 \cite{Ren2007} further proposes leader-following consensus algorithms for second-order integrator agents.
\cite{Ma2010} and \cite{You2011} develop necessary and sufficient conditions for consensusability of discrete-time and continuous-time multi-agent systems, respectively.
Variants of these algorithms have been designed with
different communication topologies, system dynamics and uncertain features, see the works \cite{Huang2017a,Zhu2018,Xu2021,Xu2018}  and the references therein.
It is noted that the aforementioned works focus on homogeneous systems with  identical dynamics.

However, in practice, many systems has heterogeneity with different system dynamics and even state space dimension.
Under the circumstances, the traditional state consensus algorithm based on Kronecker product method is not applicable,
so the consensus problem of heterogenous multi-agent system is more challenging than the homogeneous case.
In recent years,
 the distributed feedforward control approach \cite{Su2012a,Huang2017} has been widely used in solving the consensus problem of heterogeneous multi-agent systems.
 Based on the feedforward control theory framework, researchers also further studied the robust output regulation problem \cite{Bi2022},
 leader-follower consensus for uncertain nonlinear systems subject to communication delays and switching networks \cite{Lu2018},
 and bipartite output consensus over cooperative-competition networks \cite{Yaghmaie2017}, to name just a few.
The core of this method includes two points: the design of distributed observer or compensator to estimate the leader's state, and
 solving output regulation equations.
It should be pointed that the choice of the coupling gain to ensure the stability of distributed observer is  dependent on the eigenvalues of the communication topology.
Besides, in light with the distributed internal model method \cite{Wieland2011},
some effort has also been made for the consensus of heterogeneous systems \cite{Zuo2017},
where an internal model requirement is necessary and sufficient for the synchronizability of heterogeneous agents,
and a transmission zero condition never holds when the dimension of the system output is greater than that of the system input.
 However, the aforementioned results focus on the output consensus problem and has not considered  the performance optimization.

The global optimality problem of control protocols  means that the designed consensus protocol can optimize a global cost function.
Its main difficulty is that the state information of each agent is available to all other subsystems which is hard to obtain in most applications.
Therefore, the optimality problem becomes quite complicated, and the development on the optimal consensus problem is comparatively slow.
By using linear quadratic regulator (LQR) methods based on Riccati solution,
\cite{Cao2010} derives a centralized optimal consensus protocol for single integrator multi-agent system, while the corresponding optimal topology is a complete graph rather than a distributed one.
\cite{Sun2021} designes a distributed optimal controller based on non-Riccati solution for second-order multi-agent systems.
The distributed optimal control problem for the general linear multi-agent systems has been studied in \cite{Movric2014,Chen2020,Zhang2021,Jiao2020},
but the cost function is very complicated and to some extent man-made, so there is a shortage of practical significance.
 Moreover, the distributed control protocols in the above works are merely applicable to homogeneous systems.
 Although works \cite{Kiumarsi2017,Rizvi2019,Chen2020a}
 study the optimal output synchronisation problem, but they only concern local cost functions without considering the global performance cost.

Inspired by the above analyses, the tradition consensus controller for heterogeneous systems is hard to be the global optimal.
In this paper, we study the consensus problem of  heterogeneous multi-agent systems with a novel consensus control protocol based on
 the feedforward control and LQ optimal control theory.
 Compared with the existing results, the main contributions of this work are:
 1) To overcome the heterogeneity's barriers, we first design a distributed  feedforward controller based on the state and input from neighbor agents to convert the
   error system with neighbour agents into a standard linear system.
 Then, an optimal feedback controller based on the observer incorporating each agent's historical state information is designed
  by minimizing the state consensus of different neighbour agents;
2) The proposed feedback controller is obtained by solving  Riccati equations, and the corresponding global cost function under the proposed controllers is asymptotically optimal;
 3) A unified design way is provided to handle the consensus problem for heterogeneous system and homogeneous system.
 In particular, it does not require an additional reference generator, and removes the solvability assumption of output regulation equations.
 Moreover, the convergence speed is proved to be faster than that of the traditional methods.


The following notations will be used throughout this paper:
$\mathbb{R}^{n\times m}$ represents the set of $n\times m$-dimensional real matrices.
$I_{p}$ is the identity matrix with  dimension $p\times p$.
$\mbox{diag}\{a_1,a_2,\cdots,a_{N}\}$ denotes the diagonal matrix with diagonal elements being $a_1,\cdots,a_{N}$.
$\|x\|$ is the 2-norm of a vector $x$.
$\rho(A)$ is the spectral radius of matrix $A$.
$A^{T}$ and $A^{\dag}$  denote the transpose and the Moore-Penrose inverse of a matrix $A$.
$Rang(A)$ denotes the range space of $A$.

Let the interaction among $N+1$ agents be described by a directed graph
$\mathcal{G}=\{\mathcal{V},\mathcal{E},\mathcal{A}\}$,
where $\mathcal{V}=\{0,1,2,\cdots,N\}$ is the set of vertices (nodes),
$\mathcal{E}\subseteq \mathcal{V}\times \mathcal{V}$ is the set of edges.
Here, the node $0$ is associated with the leader system,
and the nodes $i=1,\cdots,N$ are associated with the followers systems.
$\mathcal{A}=[a_{ij}]\in  \mathbb{R}^{(N+1)\times (N+1)}$ is the weight matrix of $\mathcal{G}$,
$a_{ij}\neq 0$ if and only if the edge $(i,j)\in \mathcal{E}$, i.e.,
agent $j$ can use the state or the input of agent $i$ for control.
The neighbor of agent $i$ is denoted by $\mathcal{N}_{i}=\{j|(i,j)\in \mathcal{E}\}$.
A directed graph has a spanning tree if there exists a node called root such that there exists a
 directed path from this node to every other node.


\section{Problem Formulation} \label{sec:Formulation}

 We consider a  heterogeneous  discrete-time multi-agent system consisting of $N$ agents over a directed graph $\mathcal{G}$ with the dynamics of each agent given by
\begin{align}\label{multi-agent-system}
x_{i}(k+1)=A_{i}x_{i}(k)+B_{i}u_{i}(k), i=1,2,\cdots,N
\end{align}
where $x_{i}(k)\in \mathbb{R}^{n}$ and $u_{i}(k)\in \mathbb{R}^{m}$ are the state and the input of each agent.
$A_{i}\in \mathbb{R}^{n\times n}$ and $B_{i}\in \mathbb{R}^{n\times m}$ are the coefficient matrices, respectively.

The dynamic of the leader is given by
\begin{align}\label{leader}
x_{0}(k+1)=A_{0}x_{0}(k),
\end{align}
where $x_{0}\in \mathbb{R}^{n\times n}$ is the state of the leader, and $A_{0}\in \mathbb{R}^{n\times n}$ is the system matrix.

The cost function of the multi-agent system  \eqref{multi-agent-system} is given by
\begin{align}\label{cost-function}
J(s,\infty)&=\sum_{k=s}^{\infty}\left(\sum_{i=1}^{N} \sum_{j\in \mathcal{N}_{i}}(x_{i}(k)-x_{j}(k))^{T}Q(x_{i}(k)-x_{j}(k))
\right.
\notag\\
&\quad \left.+ \sum_{i=1}^{N} u_{i}^{T}(k)R_{i}u_{i}(k)\right),
\end{align}
where $Q\geq 0$ and $R_{i}>0$ are weighting matrices.

We aim to design a distributed control protocol $u_{i}(k)$
based on the available local information  of the multi-agent system \eqref{multi-agent-system}-\eqref{leader} to achieve the leader-following consensus, i.e.,
i.e., for any initial conditions $x_{i0}$,
\begin{align*}
\lim_{k\to\infty}\|x_{i}(k)-x_{0}(k)\|=0,\quad i=1,\cdots,N.
\end{align*}

\begin{assumption}\label{graph-assumption}
The directed graph $\mathcal{G}$ has a spanning tree with the node $0$ as the root.
\end{assumption}

\section{Main Results}\label{sec:main-results}
\subsection{State consensus of heterogeneous multi-agent systems}\label{sec:state-consensus}
Define the relative state error variables between agent $i$ and $j$ as:
\begin{align*}
e_{ij}(k)=x_{i}(k)-x_{j}(k).
\end{align*}
Then, the state error system is given by
\begin{align}\label{neighbour-error-system}
e_{ij}(k+1)&=A_{i}e_{ij}(k)+B_{i}u_{i}(k)
\notag\\
&\quad +(A_{i}-A_{j})x_{j}(k)-B_{j}u_{j}(k)
\end{align}
with $i=1,\cdots,N$, $j\in \mathcal{N}_{i}$ and $B_{0}=0$.

Let the controller for the $i$-th agent be given by
\begin{align}\label{distributed-controller}
u_{i}(k)=\tilde{u}_{i}(k)+\bar{u}_{i}(k),
\end{align}
we can view the last two parts in \eqref{neighbour-error-system} as a ``disturbance" term $d_{i}(k)=(A_{i}-A_{j})x_{j}(k)-B_{j}u_{j}(k)$,
to compensate $d_{i}(k)$, the first part feedforward controller  $\tilde{u}_{i}(k)$ in \eqref{distributed-controller} is designed as:
\begin{itemize}
  \item Case 1: $B_{i}$ is an invertible matrix, then
   \begin{align}\label{feedforword-controller}
\tilde{u}_{i}(k)=-B_{i}^{-1}(A_{i}-A_{j})x_{j}(k)+B_{i}^{-1}B_{j}u_j(k),
\end{align}
  \item Case 2: $Rang([A_{j}-A_{i}, B_{j}])\subseteq Rang (B_{i})$, then
\end{itemize}
\begin{align} \label{feedforword-controller-case2}
\tilde{u}_{i}(k)&=-B_{i}^{\dagger}[(A_{i}-A_{j})x_{j}(k)+B_{j}u_j(k)]
\notag\\
&\quad +(I-B_{i}^{\dagger}B_{i})z(k), \mbox{~for~} \forall z(k).
\end{align}
It's clear that $\tilde{u}_{i}(k)$ only needs the state and input information from the neighbor agent $j$, so it is a distributed form.
Under the feedforward control \eqref{feedforword-controller} or \eqref{feedforword-controller-case2}, the relative neighbor error system \eqref{neighbour-error-system}
is rewritten as the following standard linear system:
\begin{align}\label{standard-error-system}
e_{ij}(k+1)&=A_{i}e_{ij}(k)+B_{i}\bar{u}_{i}(k).
\end{align}

Since the consensus for heterogeneous multi-agent systems  \eqref{multi-agent-system} and \eqref{leader} is equivalent to the stability of the state error system \eqref{neighbour-error-system}.
Even though the total edge number over the communication graph could be more than $N$, we only need to select edges along a spanning tree.
Therefore, the corresponding global error system is written as
\begin{align}\label{gloabl-closed-error-system}
E(k+1)=\tilde{A}E(k)+\tilde{B}\bar{u}(k),
\end{align}
where $E(k)=\begin{bmatrix}
e_{10}^{T}(k) & \cdots & e_{ij}^{T}(k) \cdots & e_{N,N-1}^{T}(k)
\end{bmatrix}^{T}$,
$\bar{u}(k)=\begin{bmatrix}
\bar{u}_{1}^{T}(k) & \bar{u}_{2}^{T}(k) \cdots & \bar{u}_{N}^{T}(k)
\end{bmatrix}^{T}$,
$\tilde{A}=diag\{A_{1};A_{2};\cdots;A_{N}\}$,
and $\tilde{B}=diag\{B_{1};B_{2};\cdots;B_{N}\}$.

In this case, the cost function \eqref{cost-function} is reformulated as
\begin{align}\label{cost-function-error}
J(s,\infty)&=\sum_{k=s}^{\infty}\left(\sum_{i=1}^{N} \sum_{j\in \mathcal{N}_{i}}e_{ij}^{T}(k)Qe_{ij}(k)+
 \sum_{i=1}^{N} \bar{u}_{i}^{T}(k)R_{i}\bar{u}_{i}(k)\right)
\notag\\
&=\sum_{k=s}^{\infty}[E^{T}(k)\mathcal{Q}E(k)+\bar{u}^{T}(k)R\bar{u}(k)],
\end{align}
with $\mathcal{Q}=I_{N}\otimes Q$ and $R=diag\{R_{1},\cdots, R_{N}\}$.

\begin{remark}
The optimal consensus problem of heterogeneous multi-agent system \eqref{multi-agent-system},\eqref{leader} with performance cost \eqref{cost-function}
is converted to the optimal control problem of the global error system \eqref{gloabl-closed-error-system} with cost function \eqref{cost-function-error}.
Moreover, the cost function \eqref{cost-function} or its variant \eqref{cost-function-error} considered in this paper
is quite straightforward and general in contrast to earlier works \cite{Movric2014,Chen2020,Jiao2020}, where the state weight matrix $Q$  is assumed
to be of a special form, containing the global information of communication topology.
\end{remark}

The  key task is to determine the feedback controller $\bar{u}_{i}(k)$ based on the LQ optimal control theory.
In the ideal (complete graph) case, the error information $E(k)$ is available for all agents,
 the solvability of the optimal control problem for system \eqref{gloabl-closed-error-system} with the cost function \eqref{cost-function-error} is
 equivalent to the following standard LQ optimal control problem \cite{Anderson1971}.

\begin{lemma}\cite{Anderson1971}
Suppose that $E(k)$ is available for all agents.
Consider system \eqref{gloabl-closed-error-system} with cost \eqref{cost-function-error}.
 Then, the optimal controller is given
\begin{align}\label{centralized-optimal-control}
\bar{u}^{*}(k)=KE(k),
\end{align}
where the feedback gain $K$ is given by
\begin{align}\label{feedback-gain-matrix}
K=-(R+\tilde{B}^{T}P\tilde{B})^{-1}\tilde{B}^{T}P\tilde{A}
\end{align}
and $P$ is the solution of the following  ARE
\begin{align}\label{algebra-riccati-equation}
P=\tilde{A}^{T}P\tilde{A}+\mathcal{Q}-\tilde{A}^{T}P\tilde{B}(R+ \tilde{B}^{T}P\tilde{B})^{-1}\tilde{B}^{T}P\tilde{A}.
\end{align}
The corresponding optimal cost function is
\begin{align}\label{optimal-cost-function}
J^{*}(s,\infty)=E^{T}(s)PE(s).
\end{align}
Moreover, if $P$ is the unique positive definite solution to \eqref{algebra-riccati-equation},
then $\tilde{A}+\tilde{B}K$ is stable.
\end{lemma}
It's obvious that multi-agent system \eqref{multi-agent-system} and \eqref{leader}
can achieve leader-follower consensus under the centralized optimal controller \eqref{centralized-optimal-control}.
However, when $E(k)$ is not available for all agents, we will design a distributed controller based an observer  incorporating  agent's historical state information.

To proceed further, the system \eqref{gloabl-closed-error-system} is rewritten as
\begin{align}\label{gloabl-error-system}
E(k+1)
&=\tilde{A}E(k)+\sum_{i=1}^{N}\tilde{B}_{i}\bar{u}_{i}(k),
\\
Y_{i}(k)&=H_{i}E(k), i=1,\cdots,N
\end{align}
where $\tilde{B}_{i}=\begin{bmatrix}
0 & \cdots & B^{T}_{i} & 0\cdots & 0
\end{bmatrix}^{T}$, $Y_{i}(k)$ is measurement, and $H_{i}$  is composed of $0$ and $I_{n}$,  whose specific forms depend on the interaction among agents.

Next, we design a new distributed feedback controller for agent $i$ as
\begin{align}\label{feedback-controller}
\bar{u}_{i}^{*}(k)=K_{i}\hat{E}_{i}(k),
\end{align}
where $\hat{E}_{i}(k)$ is a distributed observer based on the available information of agent $i$ to estimate the global error $E(k)$,
\begin{subequations}\label{observer-design}
\begin{align}
\hat{E}_{1}(k+1)&=\tilde{A}\hat{E}_{1}(k)+\tilde{B}_{1}\bar{u}_{1}^{*}(k)+\tilde{B}_{2}K_{2}\hat{E}_{1}(k)+\cdots
\notag\\
&\quad+\tilde{B}_{N}K_{N}\hat{E}_{1}(k)+ L_{1}(Y_{1}(k)-H_{1}\hat{E}_{1}(k)),
\\
 &\quad \cdots \quad \cdots \quad \cdots
 \notag\\
\hat{E}_{i}(k+1)&=\tilde{A}\hat{E}_{i}(k)+\tilde{B}_{1}K_{1}\hat{E}_{i}(k) +\cdots
\notag\\
&\quad+ \tilde{B}_{i-1}K_{i-1}\hat{E}_{i}(k) +\tilde{B}_{i} \bar{u}_{i}^{*}(k)
\notag\\
&\quad +\tilde{B}_{i+1}K_{i+1}\hat{E}_{i}(k) +\cdots+\tilde{B}_{N}K_{N} \hat{E}_{i}(k)
\notag\\
&\quad +L_{i}(Y_{i}(k)-H_{i}\hat{E}_{i}(k)),
 \\
& \quad \cdots \quad \cdots \quad \cdots
 \notag \\
\hat{E}_{N}(k+1)&=\tilde{A}\hat{E}_{N}(k)+\tilde{B}_{1}K_{1}\hat{E}_{N}(k) +\cdots
\notag\\
&\quad+ \tilde{B}_{N-1}K_{N-1}\hat{E}_{N}(k) +\tilde{B}_{N} \bar{u}_{N}^{*}(k)
\notag\\
&\quad +L_{N}(Y_{N}(k)-H_{N}\hat{E}_{N}(k)),
\end{align}
\end{subequations}
where $K_i=\begin{bmatrix} 0 & \cdots & I & 0 \cdots &0
\end{bmatrix}K$, which is obtained by solving ARE \eqref{algebra-riccati-equation},
 $L_{i}$ is to be determined later to ensure the stability of the observers.

\begin{theorem}\label{main-result-1}
Let Assumption   \ref{graph-assumption} hold.
Consider the global error system \eqref{gloabl-error-system},
and the distributed control laws \eqref{feedback-controller}-\eqref{observer-design}.
 If there exist observer gains $L_{i},i=1,\cdots,N$ such that the matrix
\begin{align}\label{closed-loop-Ac}
A_{c}=\begin{bmatrix}
\Theta_1 & -\tilde{B}_{2}K_{2} &  \cdots & -\tilde{B}_{N}K_{N}\\
-\tilde{B}_{1}K_{1} & \Theta_2 &   \cdots & -\tilde{B}_{N}K_{N} \\
 \vdots &  \vdots & \ddots  & \vdots \\
-\tilde{B}_{1}K_{1} &  \cdots  & -\tilde{B}_{N-1}K_{N-1} &  \Theta_N
\end{bmatrix}
\end{align}
is asymptotically stable, where $\Theta_{i}=\tilde{A}+\tilde{B}K-\tilde{B}_{i}K_{i}-L_{i}H_{i}$.
Then the observers \eqref{observer-design} are asymptotically stable, i.e.,
\begin{align}\label{observer-error-vector}
\lim_{k\to \infty}\|\hat{E}_{i}(k)-E(k)\|=0.
\end{align}
Moreover,  if the Riccati equation \eqref{algebra-riccati-equation} has a positive definite solution $P$,
 under the distributed controllers \eqref{distributed-controller} and \eqref{feedback-controller},
the multi-agent system \eqref{multi-agent-system} and \eqref{leader}  can achieve leader-follower consensus.
 \end{theorem}

\begin{proof}
Define the observer error vector
$\tilde{E}_{i}(k)=E(k)-\hat{E}_{i}(k)$.
Then, combining system  \eqref{gloabl-error-system} with observers \eqref{observer-design}, one obtains
\begin{align}\label{gloabl-error-cloosed-loop-system}
E(k+1)
&=(\tilde{A}+\tilde{B}K)E(k)-\tilde{B}_{1}K_{1}\tilde{E}_{1}(k)
\notag\\
&-\tilde{B}_{2}K_{2}\tilde{E}_{2}(k)-\cdots-\tilde{B}_{N}K_{N}\tilde{E}_{N}(k),
\end{align}
\begin{subequations}\label{observer-error-system}
\begin{align}
\tilde{E}_{1}(k+1)&=(\tilde{A}+\tilde{B}K-\tilde{B}_{1}K_{1}-L_{1}H_{1})\tilde{E}_{1}(k)
\notag\\
&-\tilde{B}_{2}K_{2}\tilde{E}_{2}(k)-\cdots-\tilde{B}_{N}K_{N}\tilde{E}_{N}(k),
\\
& \quad \cdots \quad \cdots \quad \cdots
\notag\\
\tilde{E}_{i}(k+1)&=(\tilde{A}+\tilde{B}K-\tilde{B}_{i}K_{i}-L_{i}H_{i})\tilde{E}_{i}(k)
\notag\\
&\quad -\tilde{B}_{1}K_{1}\tilde{E}_{1}(k)-\cdots-\tilde{B}_{i-1}K_{i-1}\tilde{E}_{i-1}(k)
\notag\\
&\quad-\tilde{B}_{i+1}K_{i+1}\tilde{E}_{i+1}(k)-\cdots-\tilde{B}_{N}K_{N}\tilde{E}_{N}(k),
\\
& \quad \cdots \quad \cdots \quad \cdots
\notag\\
\tilde{E}_{N}(k+1)&=(\tilde{A}+\tilde{B}K-\tilde{B}_{N}K_{N}-L_{N}H_{N})\tilde{E}_{N}(k)
\notag\\
&-\tilde{B}_{1}K_{1}\tilde{E}_{1}(k)-\cdots-\tilde{B}_{N-1}K_{N-1}\tilde{E}_{N-1}(k).
\end{align}
\end{subequations}

According to \eqref{observer-error-system},
we have
\begin{align}\label{obsever-error-system-global}
\tilde{E}(k+1)=A_{c}\tilde{E}(k),
\end{align}
where
$\tilde{E}(k)=\begin{bmatrix}\tilde{E}_{1}^{T}(k), \tilde{E}_{2}^{T}(k),\cdots,\tilde{E}_{N}^{T}(k) \end{bmatrix}^{T}$.
Obviously, if there exist matrices $L_{i}$ such that $A_{c}$ is asymptotically stable, then observer errors $\tilde{E}(k)$ converge to zero as $k\to \infty$,
i.e., Eq. \eqref{observer-error-vector} holds.
Furthermore,
it follows from \eqref{gloabl-error-cloosed-loop-system} and \eqref{observer-error-system} that
\begin{align}\label{closed-loop-error-system}
\begin{bmatrix}
E(k+1)\\
\tilde{E}(k+1)
\end{bmatrix}
=\bar{A}_{c} \begin{bmatrix}
E(k)\\
\tilde{E}(k)
\end{bmatrix}
\end{align}
where $\bar{A}_c=\begin{bmatrix}
\tilde{A}+\tilde{B}K &  \Psi\\
 0 & A_{c}
\end{bmatrix}$
and $\Psi=\begin{bmatrix}
 -\tilde{B}_{1}K_{1}& \cdots & -\tilde{B}_{N}K_{N}
\end{bmatrix}$.
Since $P$ is the positive definite solution to Riccati equation \eqref{algebra-riccati-equation},  then $\tilde{A}+\tilde{B}K$ is stable,
 based on the LQ control theory, the leader-follower consensus of multi-agent system \eqref{multi-agent-system} can be achieved, i.e.,
 $\lim_{k\to \infty} \|x_{i}(k)-x_{0}(k)\|=0$ for $i=1,2,\cdots,N$.

The proof is completed.

\end{proof}

According to the consensus error system \eqref{closed-loop-error-system},
the convergence speed of $N+1$ agents relies on the spectral radius of $\bar{A}_c$, which is determined by $\tilde{A}+\tilde{B}K$ and $A_{c}$.
Since the feedback gain matrix  $K$ has been given by \eqref{feedback-gain-matrix},
 to accelerate the convergence speed, we need to choose the  optimal observer gain matrices $L_{i},i=1,2,\cdots,N$ such that the spectral radius of
$A_{c}$  is as small as possible.
 we can determine $L_{i}$ by solving an   optimization problem.

To this end, denote
\begin{align*}
A_{w}&=\begin{bmatrix}
A_{K}-\tilde{B}_{1}K_{1}& -\tilde{B}_{2}K_{2} &  \cdots & -\tilde{B}_{N}K_{N}\\
-\tilde{B}_{1}K_{1} & A_{K}-\tilde{B}_{2}K_{2} &   \cdots & -\tilde{B}_{N}K_{N} \\
 \vdots &  \vdots & \ddots  & \vdots \\
-\tilde{B}_{1}K_{1} &  \cdots  & \cdots&  A_{K}-\tilde{B}_{N}K_{N}
\end{bmatrix},
\notag\\
A_{K}&=\tilde{A}+\tilde{B}K, \quad
H=diag
\{H_{1}, H_{2}, \cdots, H_{N}\},
\notag\\
P_{c}&=diag
\{P_{c1}, P_{c2}, \cdots, P_{cN}\},\quad
Y=diag
\{Y_{1}, Y_{2}, \cdots, Y_{N}\},
\notag\\
L&=diag
\{L_{1}, L_{2}, \cdots, L_{N}\}.
\end{align*}
\begin{lemma}\label{optimization-lemma}
If there exist matrices $P_{c},Y,S$ and parameter $\alpha$ such that
\begin{subequations}\label{Lmi}
\begin{align}
S=S^{T}>0,
\\
P_{c}=P_{c}^{T}>0,
\\
\alpha I-P_{c}>0,
\\
\begin{bmatrix}
P_{c}-S & (P_{c}A_{w}-YH)^{T}\\
P_{c}A_{w}-YH & P_{c}
\end{bmatrix}>0,\label{lmi-solving}
\end{align}
\end{subequations}
and select the observer gains $L_{i}=P_{ci}^{-1}Y_{i}$.
In this case, the observer gains  are solved by the following optimization problem
\begin{align}\label{choose-observer-gain-L}
\min_{L_{i},P_{c},S} \alpha \quad  \quad \mbox{s.t.} \quad \eqref{Lmi}
\end{align}
\end{lemma}

\begin{proof}
The observer error system \eqref{obsever-error-system-global} is rewritten as
\begin{align}\label{observer-error-system-jizhongshi}
\tilde{E}(k+1)=A_{c}\tilde{E}(k)=(A_{w}-LH)\tilde{E}(k).
\end{align}
\eqref{observer-error-system-jizhongshi} is asymptotically stable if
 there exist a symmetric positive definite matrix $P_{c}$ satisfying the Lyapunov inequality
\begin{align}
(A_{w}-LH)^{T}P_{c}(A_{w}-LH)-P_{c}<0.
\end{align}
Therefore, there exists a positive definite matrix $S^{T}=S>0$ such that
\begin{align}\label{lmi-s}
(A_{w}-LH)^{T}P_{c}(A_{w}-LH)-P_{c}<-S.
\end{align}
Taking $L=P_{c}^{-1}Y$, the inequality \eqref{lmi-s} is equivalent to \eqref{lmi-solving}.
Now, we consider the cost function
\begin{align}\label{observer-cost-function}
J_{L}=\sum_{k=0}^{T_{N}} \tilde{E}(k)^{T}S\tilde{E}(k)
\end{align}
By taking into consideration of \eqref{lmi-s}, we can derive
\begin{align*}
&\tilde{E}(k+1)^{T}P_{c}\tilde{E}(k+1)-\tilde{E}(k)^{T}P_{c}\tilde{E}(k)
\notag\\
&=\tilde{E}(k)^{T}(A_{c}^{T}P_{c}A_{c}-P_{c})\tilde{E}(k)
 <-\tilde{E}(k)^{T}S\tilde{E}(k).
\end{align*}
Then,
\begin{align*}
J_{L}\leq \tilde{E}^{T}(0)P_{c}\tilde{E}(0)-\tilde{E}^{T}(T_{N})P_{c}\tilde{E}(T_{N})
 \leq \tilde{E}^{T}(0)P_{c}\tilde{E}(0)
\end{align*}
with $T_{N}\to \infty$, $\tilde{E}^{T}(T_{N})P_{c}\tilde{E}(T_{N})\to 0$. This inequality implies that
one can minimize the cost function \eqref{observer-cost-function} by minimizing the bound $E^{T}(0)P_{c}E(0)$.
Since $\tilde{E}^{T}(0)P_{c}\tilde{E}(0)\leq \|\tilde{E}(0)\|^{2}\|P_{c}\|\leq M_{0}^{2}\|P_{c}\|$, where $M_{0}$ is the upper bound of the initial value $ \tilde{E}(0)$.
Therefore, the optimal observer gain $L$ is derived by minimizing the maximum eigenvalue of $P_{c}$,
i.e., solving the minimization problem \eqref{choose-observer-gain-L}.
\end{proof}

\begin{remark}
Lemma \ref{optimization-lemma} provides an approach to choose gain matrices $L_{i}$ such that the spectral radius of $A_{c}$  is as small as possible.
Due to the constraints on the diagonal structure of $P_{c}$ and $Y$,  the derived observer gain $L_{i}$ is a suboptimal solution.
\end{remark}

Next, we will derive the cost difference between the proposed distributed controller \eqref{feedback-controller},
and the centralized optimal control \eqref{centralized-optimal-control},
and then analyze the asymptotical optimal property of the corresponding cost function.

For the convenience of analysis, denote that
\begin{align*}
\mathcal{M}_{1}&=(\tilde{A}+\tilde{B}K)^{T}P\Psi
-\begin{bmatrix}
K_{1}^{T}R_{1}K_{1} &  \cdots & K_{N}^{T}R_{N}K_{N}
\end{bmatrix},
\notag\\
\mathcal{M}_{2}&=\begin{bmatrix}
K_{1}^{T}R_{1}K_{1} & \cdots & 0\\
0  & \ddots & 0\\
\vdots &  \vdots & \vdots\\
0  &\cdots & K_{N}^{T}R_{N}K_{N}
\end{bmatrix}+\Psi^{T}P\Psi.
\end{align*}

\begin{theorem}\label{asymptotically-optimal-theorem}
Under the proposed distributed controllers \eqref{feedback-controller} and \eqref{observer-design} with $L_{i},i=1,2,\cdots,N$
such that the matrix $A_c$ in  \eqref{closed-loop-Ac}is asymptotically stable,
the corresponding cost function  is given by
\begin{align}\label{cost-function-distributed}
&J^{\star}(s,\infty)
\notag\\
&=E^{T}(s)PE(s)+\sum_{k=s}^{\infty}\begin{bmatrix}
E(k)\\
\tilde{E}(k)
\end{bmatrix}^{T}
\begin{bmatrix}
0 & \mathcal{M}_{1}\\
\mathcal{M}_{1}^{T} & \mathcal{M}_{2}
\end{bmatrix}
\begin{bmatrix}
E(k)\\
\tilde{E}(k)
\end{bmatrix}
\end{align}
Moreover, the cost difference between the cost function \eqref{cost-function-distributed} and the cost under the centralized optimal control
is given by
\begin{align}\label{cost-difference}
 \Delta J(s,\infty)&=J^{\star}(s,\infty)-J^{*}(s,\infty)
 \notag\\
 &=\sum_{k=s}^{\infty}\begin{bmatrix}
E(k)\\
\tilde{E}(k)
\end{bmatrix}^{T}
\begin{bmatrix}
0 & \mathcal{M}_{1}\\
\mathcal{M}_{1}^{T} & \mathcal{M}_{2}
\end{bmatrix}
\begin{bmatrix}
E(k)\\
\tilde{E}(k)
\end{bmatrix}.
\end{align}
In particular, the optimal cost function difference will approach to zero as $s$ is sufficiently large.
That is to say, the proposed  consensus controller can achieve the optimal cost (asymptotically).
\end{theorem}

\begin{proof}
According to \eqref{algebra-riccati-equation} and \eqref{gloabl-error-cloosed-loop-system},
  we have
  \begin{align*}
  &E^{T}(k)PE(k)-E^{T}(k+1)PE(k+1)
  \notag\\
  &=E^{T}(k)(\mathcal{Q}+K^{T}RK)E(k)
\notag\\
  &\quad -\tilde{E}^{T}(k)\Psi^{T}P (\tilde{A}+\tilde{B}K)E(k)
  \notag\\
  &\quad -E^{T}(k)(\tilde{A}+\tilde{B}K)^{T}P\Psi\tilde{E}(k)
  -\tilde{E}^{T}(k)\Psi^{T}P\Psi \tilde{E}(k)
  \end{align*}
  Based on the cost function \eqref{cost-function} under the centralized optimal control,
   by performing summation on $k$ from $s$ to $\infty$ and applying algebraic calculations yields
  \begin{align}
  &E^{T}(s)PE(s)-E^{T}(\infty)PE(\infty)
  \notag\\
  &=J(s,\infty)-\sum_{k=s}^{\infty}\begin{bmatrix}
E(k)\\
\tilde{E}(k)
\end{bmatrix}^{T}
\begin{bmatrix}
0 & \mathcal{M}_{1}\\
\mathcal{M}_{1}^{T} & \mathcal{M}_{2}
\end{bmatrix}
\begin{bmatrix}
E(k)\\
\tilde{E}(k)
\end{bmatrix}
  \end{align}
 It follows from Theorem \ref{main-result-1} that
 $\lim_{k\to \infty} E^{T}(k)PE(k)=0$.
 Therefore,
 the corresponding optimal cost function $J^{\star}(s,\infty)$ under the  proposed distributed observer-based controller \eqref{feedback-controller}
  is given by \eqref{cost-function-distributed}.
In line with \eqref{optimal-cost-function},
 the optimal cost difference \eqref{cost-difference} holds.
 Furthermore, we can analyze the asymptotically optimality of the distributed controller \eqref{feedback-controller},
 which is similar to that in \cite{Zhang2023}.
 To avoid repetition, we omit the detail in this paper.
This proof is completed.
\end{proof}

\subsection{Comparison with traditional consensus algorithms}
Firstly, the consensus with the proposed controller has a faster convergence speed than the traditional consensus algorithms.

In fact,
from the closed-loop system \eqref{closed-loop-error-system},
  one has
\begin{align}\label{norm-estimator-error}
\left\|\begin{bmatrix}
E(k)\\
\tilde{E}(k)
\end{bmatrix}\right\|\leq \rho(\bar{A}_{c})\left\|\begin{bmatrix}
E(k-1)\\
\tilde{E}(k-1)
\end{bmatrix}\right\|,
\end{align}
where $\rho(\bar{A}_{c})$ is the larger spectral radius of  $\tilde{A}+\tilde{B}K$ and $A_{c}$.
In particularly,
 $\tilde{A}+\tilde{B}K$ is the closed-loop system matrix obtained by the optimal feedback control \eqref{feedback-gain-matrix},
 that is,
 $E(k+1)=(\tilde{A}+\tilde{B}K)E(k)$ while $E(k+1)^{T} \mathcal{Q}E(k+1)$ is minimized as in \eqref{cost-function},
so the modulus of the eigenvalues for $\tilde{A}+\tilde{B}K$ is minimized in certain sense.
Besides, based on  the optimization in Lemma \ref{optimization-lemma}, we can appropriately select $L_{i}$
such that the upper bound of the spectral radius $\rho( A_{c})$ is as small as possible.
From these perspectives, $\rho(\bar{A}_{c})$ is made more to be small. This is in comparison with the conventional consensus algorithms where
 the maximum eigenvalue of the matrix $\bar{A}_{c}$ is not minimized and is determined by the eigenvalues of the Laplacian matrix $\mathcal{L}$.
Therefore, it can be expected that the proposed approach can achieve a faster convergence than the conventional algorithms
 as demonstrated in the simulation examples in Section \ref{sec:example}.


Second, the cost difference $\Delta J(s,\infty)$ between the new distributed controller \eqref{feedback-controller} and the centralized optimal control \eqref{centralized-optimal-control}
is provided in Theorem \ref{asymptotically-optimal-theorem},
and it is equal to zero as $s\to\infty$.
 That is to say,  the corresponding cost function under the proposed distributed controllers \eqref{feedback-controller} is asymptotically optimal.

\begin{remark}
To break through the heterogeneity's barriers, we pre-design a feeforward controller \eqref{feedforword-controller} or \eqref{feedforword-controller-case2},
and then the remaining distributed feedback controller design
 can be directly reduced to the homogeneous case \cite{Zhang2023},
 which indicates the current consensus controller can deal with both heterogeneous system and homogeneous systems.
So it is totally different from the relevant research works \cite{Huang2017,Zhang2020,Wieland2011}.
 In the proposed approach, the solvability assumption of output regulation equations
and an additional distributed observer to estimate the leader's state have been removed.
A kind of novel distributed controllers based on a distributed  observer involving local information are developed
to solve the consensus problem of heterogeneous systems \eqref{multi-agent-system} and \eqref{leader}.
Moreover, the feedback gain matrices  $K_{i}$ are obtained by solving Riccati equations,
which do not require the calculation of eigenvalues for communication topology.
\end{remark}

\subsection{Output consensus of heterogeneous multi-agent systems}\label{sec:output-consensus}

Different from the cases discussed in Section \ref{sec:state-consensus}, where the state dimension of each subsystem is identical,
we further study the output consensus of heterogeneous systems with different state space dimensions.

Consider the heterogeneous multi-agent system consisting of $N$ followers, whose dynamics are given by
\begin{align}\label{multi-agent-system-heterogeneous}
x_{i}(k+1)&=A_{i}x_{i}(k)+B_{i}u_{i}(k),
\notag\\
y_{i}(k)&=C_{i}x_{i}(k), i=1,2,\cdots,N
\end{align}
where $x_{i}(k)\in \mathbb{R}^{n_{i}}$, $u_{i}(k)\in \mathbb{R}^{m_{i}}$ and $y_{i}\in \mathbb{R}^{q}$ are the state, the input and the output of  agent $i$.
$A_{i}\in \mathbb{R}^{n_{i}\times n_{i}},B_{i}\in \mathbb{R}^{n_{i}\times m_{i}}$ and $C_{i}\in \mathbb{R}^{q\times n_{i}}$ are the coefficient matrices.

The dynamic of the leader is given by
\begin{align}\label{leader-2}
x_{0}(k+1)&=A_{0}x_{0}(k),
\notag\\
y_{0}(k)&=C_{0}x_{0}(k),
\end{align}
where $x_{0}(k)\in \mathbb{R}^{p\times p}$ and $y_{0}(k)\in \mathbb{R}^{q}$ is the state of the leader,
$A_{0}\in \mathbb{R}^{p\times p}$ and $C_{0}\in \mathbb{R}^{q\times p}$ are the coefficient matrices.

The cost function of multi-agent systems \eqref{multi-agent-system-heterogeneous} is given by
\begin{align}\label{cost-function-heterogeneous}
J(s,\infty)&=\sum_{k=s}^{\infty}\left(\sum_{i=1}^{N} \sum_{j\in \mathcal{N}_{i}}(y_{i}(k)-y_{j}(k))^{T}Q(y_{i}(k)-y_{j}(k))
\right.
\notag\\
&\quad \left.+ \sum_{i=1}^{N} u_{i}^{T}(k)R_{i}u_{i}(k)\right)
\end{align}
where $Q\geq 0$ and $R_{i}>0$ are weighting matrices.

Because the state space dimension of each subsystem is not identical,
the above state consensus is meaningless.
We need to design a distributed controller $u_{i}(k)$
to ensure that the outputs of the followers synchronize to the leader's output,
\begin{align*}
\lim_{k\to\infty}\|y_{i}(k)-y_{0}(k)\|=0,\quad i=1,2,\cdots, N.
\end{align*}
To this end, we define relative output error variables between agent $i$ and $j$ as:
\begin{align}
\varepsilon_{ij}(k)=y_{i}(k)-y_{j}(k), \quad  j\in \mathcal{N}_{i}.
\end{align}
Then, the dynamics of output error system are
\begin{align}\label{output-error-system}
\varepsilon_{ij}(k+1)&=\varepsilon_{ij}(k)+C_{i}(A_{i}-I_{n_{i}})x_{i}(k)+ C_{i}B_{i}u_{i}(k)
\notag\\
&\quad-C_{j}(A_{j}-I_{n_{j}})x_{j}(k)-C_{j}B_{j}u_{j}(k)
\end{align}
with $B_{0}=0, u_{0}(k)=0$ and $j\in \mathcal{N}_{i}$.

Similarly, we design the controller $u_{i}(k)$ as
\begin{align}\label{distributed-controller-2}
u_{i}(k)=\tilde{u}_{i}(k)+\bar{u}_{i}(k).
\end{align}
Denote that $A_{ij}=\begin{bmatrix}
-C_{i}(A_{i}-I_{i}) & C_{j}(A_{j}-I_{j})& C_{j}B_{j}
\end{bmatrix}$.
When $C_{i}B_{i}$ is an invertible matrix, then
the feedforward controller $\tilde{u}_{i}(k)$ is designed as:
\begin{align}\label{feedforword-controller-2-case1}
\tilde{u}_{i}(k)&=-B_{i}^{-1}(A_{i}-I_{n_{i}})x_{i}(k)+(C_{i}B_{i})^{-1}(C_{j}B_{j})u_{j}(k)
\notag\\
&\quad +(C_{i}B_{i})^{-1}C_{j}(A_{j}-I_{n_{j}})x_{j}(k).
\end{align}
When $Rang(A_{ij})\subseteq Rang(C_{i}B_{i})$, then
\begin{align}\label{feedforword-controller-2-case2}
\tilde{u}_{i}(k)&=-B_{i}^{\dagger}(A_{i}-I_{n_{i}})x_{i}(k)+(C_{i}B_{i})^{\dagger}(C_{j}B_{j})u_{j}(k)
\notag\\
&\quad +(C_{i}B_{i})^{\dagger}C_{j}(A_{j}-I_{n_{j}})x_{j}(k).
\end{align}
Observe that the distributed feedforward controller $\tilde{u}_{i}(k)$ is designed by using only the state information of itself and its neighbor agent $j$ and the control input of
neighbor agent $j$.
Under the feedforward controller \eqref{feedforword-controller-2-case1} or \eqref{feedforword-controller-2-case2},
the output error system \eqref{output-error-system} is converted into a standard linear system
\begin{align}\label{standard-output-linear system}
\varepsilon_{ij}(k+1)=\varepsilon_{ij}(k)+C_{i}B_{i}\bar{u}_{i}(k).
\end{align}

Since the output consensus for heterogeneous multi-agent systems  \eqref{multi-agent-system-heterogeneous} and \eqref{leader-2} is equivalent to the stability of the output error system \eqref{output-error-system}. Similar Section\ref{sec:state-consensus}, we still select edges along a spanning tree,
the corresponding global form is
\begin{align}\label{global-output-error-system}
\Upsilon(k+1)=\mathcal{A}\Upsilon(k)+\mathcal{B}\bar{u}(k),
\end{align}
where $\Upsilon(k)=\begin{bmatrix}
\varepsilon_{10}^{T}(k) & \cdots & \varepsilon_{ij}^{T}(k) \cdots & \varepsilon_{N,N-1}^{T}(k)
\end{bmatrix}^{T}$,
$\bar{u}(k)=\begin{bmatrix}
\bar{u}_{1}^{T}(k) & \bar{u}_{2}^{T}(k) \cdots & \bar{u}_{N}^{T}(k)
\end{bmatrix}^{T}$,
$\mathcal{A}=I_{N}\otimes I_{q}$,
and $\mathcal{B}=diag\{C_{1}B_{1},C_{2}B_{2},\cdots,C_{N}B_{N}\}$.

Then, the cost function \eqref{cost-function-heterogeneous} is reformulated as
\begin{align}\label{cost-function-error-heterogeneous}
J(s,\infty)&=\sum_{k=s}^{\infty}\left(\sum_{i=1}^{N} \sum_{j\in \mathcal{N}_{i}}\varepsilon_{ij}^{T}(k)Q\varepsilon_{ij}(k)+
 \sum_{i=1}^{N} \bar{u}_{i}^{T}(k)R_{i}\bar{u}_{i}(k)\right)
\notag\\
&=\sum_{k=s}^{\infty}[\Upsilon^{T}(k)\mathcal{Q}\Upsilon(k)+\bar{u}^{T}(k)R\bar{u}(k)],
\end{align}
with $\mathcal{Q}=I_{N}\otimes Q$ and $R=diag\{R_{1},\cdots, R_{N}\}$.

Next, we will focus on designing the feedback control $\bar{u}_{i}(k)$ by minimizing the performance \eqref{cost-function-error-heterogeneous}.

We rewrite the system \eqref{global-output-error-system} as
\begin{align}\label{gloabl-error-system-heterogeneous}
\Upsilon(k+1)
&=\mathcal{A}\Upsilon(k)+\sum_{i=1}^{N}\mathcal{B}_{i}\bar{u}_{i}(k),
\\
\mathcal{Y}_{i}(k)&=F_{i}\Upsilon(k), i=1,\cdots,N
\end{align}
where $\mathcal{B}_{i}=\begin{bmatrix}
0 & \cdots & (C_{i}B_{i})^{T} & 0\cdots & 0
\end{bmatrix}^{T}$, $\mathcal{Y}_{i}(k)$ is measurement,
and $F_{i}$  is composed of $0$ and $I_{n}$,  whose specific forms depend on the interaction among agents.

Next, we design a new distributed feedback controller as
\begin{align}\label{feedback-controller-2}
\bar{u}_{i}^{*}(k)=\mathcal{K}_{i}\hat{\Upsilon}_{i}(k),
\end{align}
where $\hat{\Upsilon}_{i}(k)$ is a distributed observer based on the available information of agent $i$ to estimate the global error $\Upsilon(k)$,
\begin{subequations}\label{observer-design-heterogeneous}
\begin{align}
\hat{\Upsilon}_{1}(k+1)&=\mathcal{A}\hat{\Upsilon}_{1}(k)+\mathcal{B}_{1}\bar{u}_{1}^{*}(k)+\mathcal{B}_{2}\mathcal{K}_{2}\hat{\Upsilon}_{1}(k)+\cdots
\notag\\
&\quad+\mathcal{B}_{N}\mathcal{K}_{N}\hat{\Upsilon}_{1}(k)+ \Phi_{1}(\mathcal{Y}_{1}(k)-F_{1}\hat{\Upsilon}_{1}(k)),
\\
 &\quad \cdots \quad \cdots \quad \cdots
 \notag\\
\hat{\Upsilon}_{i}(k+1)&=\mathcal{A}\hat{\Upsilon}_{i}(k)+\mathcal{B}_{1}\mathcal{K}_{1}\hat{\Upsilon}_{i}(k) +\cdots
\notag\\
&\quad+ \mathcal{B}_{i-1}\mathcal{K}_{i-1}\hat{\Upsilon}_{i}(k) +\mathcal{B}_{i} \bar{u}_{i}^{*}(k)
\notag\\
&\quad +\mathcal{B}_{i+1}\mathcal{K}_{i+1}\hat{\Upsilon}_{i}(k) +\cdots+\mathcal{B}_{N}\mathcal{K}_{N} \hat{\Upsilon}_{i}(k)
\notag\\
&\quad +\Phi_{i}(\mathcal{Y}_{i}(k)-F_{i}\hat{\Upsilon}_{i}(k)),
 \\
& \quad \cdots \quad \cdots \quad \cdots
 \notag \\
\hat{\Upsilon}_{N}(k+1)&=\mathcal{A}\hat{\Upsilon}_{N}(k)+\mathcal{B}_{1}\mathcal{K}_{1}\hat{\Upsilon}_{N}(k) +\cdots
\notag\\
&\quad+ \mathcal{B}_{N-1}\mathcal{K}_{N-1}\hat{\Upsilon}_{N}(k) +\mathcal{B}_{N} \bar{u}_{N}^{*}(k)
\notag\\
&\quad +\Phi_{N}(\mathcal{Y}_{N}(k)-F_{N}\hat{\Upsilon}_{N}(k)),
\end{align}
\end{subequations}
where $\mathcal{K}_i=\begin{bmatrix} 0 & \cdots & I & 0 \cdots &0
\end{bmatrix}\mathcal{K}$, which is obtained by solving the ARE
\begin{align}\label{algebra-riccati-equation-heterogeneous}
\mathcal{P}=\mathcal{A} ^{T}\mathcal{P}\mathcal{A}+\mathcal{Q}-\mathcal{A}^{T}P\mathcal{B}(R+ \mathcal{B}^{T}\mathcal{P}\mathcal{B})^{-1}\mathcal{B}^{T}\mathcal{P}\mathcal{A}
\end{align}
and  $\Phi_{i}$ is observer gain matrix to be determined later.

\begin{theorem}\label{main-result-2}
Let Assumption \ref{graph-assumption} hold, and consider the global error system \eqref{gloabl-error-system-heterogeneous}.
Under the distributed control laws \eqref{feedback-controller-2} with \eqref{observer-design-heterogeneous}.
 if there exist observer gains $\Phi_{i},i=1,\cdots,N$ such that the matrix
\begin{align}
H_{c}=\begin{bmatrix}
\Sigma_1 & -\mathcal{B}_{2}\mathcal{K}_{2} &  \cdots & -\mathcal{B}_{N}\mathcal{K}_{N}\\
-\mathcal{B}_{1}\mathcal{K}_{1} & \Sigma_2 &   \cdots & -\mathcal{B}_{N}\mathcal{K}_{N} \\
 \vdots &  \vdots & \ddots  & \vdots \\
-\mathcal{B}_{1}\mathcal{K}_{1}  &  \cdots  &-\mathcal{B}_{N-1}\mathcal{K}_{N-1}  &  \Sigma_N
\end{bmatrix}
\end{align}
is stable, where $\Sigma_{i}=\mathcal{A}+\mathcal{B}\mathcal{K}-\mathcal{B}_{i}\mathcal{K}_{i}-\Phi_{i}F_{i}$.
Then the observers \eqref{observer-design-heterogeneous} are stable under the controller \eqref{feedback-controller-2}, i.e.,
\begin{align}\label{observer-error-vector-1}
\lim_{k\to \infty}\|\hat{\Upsilon}_{i}(k)-\Upsilon(k)\|=0.
\end{align}
Moreover,  if the ARE \eqref{algebra-riccati-equation-heterogeneous} has a positive definite solution $\mathcal{P}$,
 under the distributed controllers  \eqref{feedforword-controller-2-case1} or \eqref{feedforword-controller-2-case2}, \eqref{feedback-controller-2} and \eqref{observer-design-heterogeneous},
the heterogeneous multi-agent systems \eqref{multi-agent-system-heterogeneous} and \eqref{leader-2}  can achieve output consensus.
Meanwhile, the corresponding cost function is also asymptotically optimal.
 \end{theorem}
\begin{proof}
This proof is similar to that of Theorem \ref{main-result-1} and Theorem \ref{asymptotically-optimal-theorem}. So the details are omitted.
\end{proof}

\begin{remark}
A new type of distributed control protocol is designed to achieve the consensus of heterogeneous multi-agent systems.
In our work, the key  techniques is to convert the relative error system \eqref{neighbour-error-system} or \eqref{output-error-system}
 into a standard linear system \eqref{standard-error-system} or \eqref{standard-output-linear system} through a feedforward control method;
 another key point is that the distributed-observer-based controllers \eqref{feedback-controller} or \eqref{feedback-controller-2} are designed by LQ optimal control theory based on available neighbour error information.
 Besides, one point should be noted that the  proposed new consensus control protocol is asymptotically optimal and is independent of the eigenvalues of communication topology.
\end{remark}

\section{Numerical Simulation}\label{sec:example}

In this section, we validate the the proposed theoretical
results through the following numerical example.

\begin{example}
\begin{figure}[!htbp]
\centering
\includegraphics[width=1.4in]{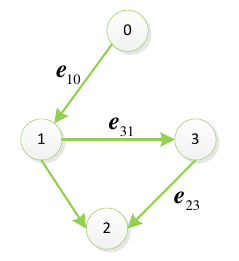}
\caption{Communication topology among four agents}
\label{fig:1-0}
\end{figure}
Consider the multi-agent system consisting of  three
heterogeneous follower agents with the system matrices:
\begin{align*}
A_{1}&=
\begin{bmatrix}
1.1 & 1\\
-2 & 1
\end{bmatrix},
A_{2}=
\begin{bmatrix}
-2 & 1  \\
1 & -1
\end{bmatrix},
A_3=\begin{bmatrix}
0.5  & 0  \\
-2 & -0.6
\end{bmatrix}
\notag\\
B_{1}&=\begin{bmatrix}
 0   & 1  \\
0.4   &  0.2
\end{bmatrix},
B_{2}=\begin{bmatrix}
 0.1   &  0.2 \\
 1   &  0.5
\end{bmatrix},
B_{3}=\begin{bmatrix}
 0.2   &  0 \\
 0.1   &  1
\end{bmatrix}
\end{align*}
The leader's state trajectory is generated by a sinusoidal trajectory generator with dynamics given by
\begin{align*}
A_{0}=\begin{bmatrix}
cos(0.5) & sin(0.5)\\
-sin(0.5) & cos(0.5)
\end{bmatrix}.
\end{align*}
The interactions of agents are given in Fig.\ref{fig:1-0}, which  satisfies  Assumption \ref{graph-assumption}.
 Each agent only receives neighbor error information, so we can
determine $H_i$ as
\begin{align*}
H_{1}=\begin{bmatrix}
 I_{2} & 0 & 0\\
 0 & 0 & I_{2}
\end{bmatrix},
H_{2}=\begin{bmatrix}
 0 & I_{2} & 0
\end{bmatrix},
H_{3}=\begin{bmatrix}
0 & I_{2} & 0\\
 0 & 0 & I_{2}
\end{bmatrix}.
\end{align*}
We choose
\begin{align*}
Q= R_{1}=R_{2}=R_{3}=I_{2}.
\end{align*}
According to ARE  \eqref{algebra-riccati-equation} and Lemma \ref{optimization-lemma},
 the feedback gains $K_{i}$ and the observer gain $L_{i}$ can be obtained, respectively.
 Fig.\ref{fig:1-2} shows that the observer error vectors $\tilde{E}_{i}(k)$ under the proposed controller \eqref{feedback-controller} converge to zero.
 The trajectories of the first state of the three followers and the leader are given in Fig.\ref{fig:1-3},
 and the consensus errors of the first state are displayed in Fig.\ref{fig:1-4},
  which show that the states of followers can synchronize with the state generated by the leader after 10 steps.
With the same initial conditions, applying the existing consensus algorithm in \cite{Huang2017},
Fig.\ref{fig:1-5}  and Fig.\ref{fig:1-6}  show the first state trajectories for four agents and the state error trajectories between agents and the leader, respectively.
To further compare the consensus speed with the traditional method, we calculate the spectral radius as $\rho(A_{c})=0.600$ by our method,
and according to the existing method,  the spectral radius is $\rho(A_{c})=0.9651$.
Therefore, the proposed distributed controller \eqref{distributed-controller} with \eqref{feedforword-controller},\eqref{feedback-controller} and \eqref{observer-design} can make all agents converge  to the leader with a faster convergence speed.


\begin{figure}[!htbp]
\centering
\includegraphics[width=2.5in]{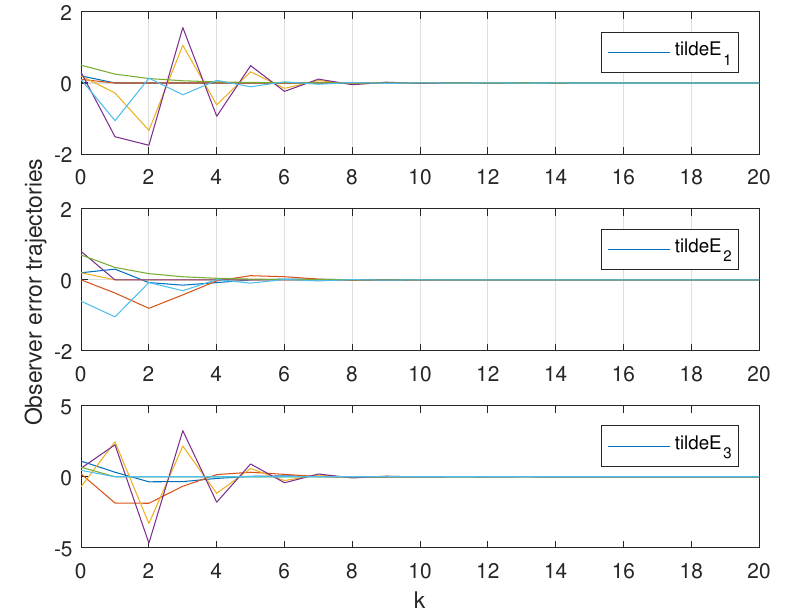}
\caption{Observer error trajectories $\tilde{E}_{i}(k)$.}
\label{fig:1-2}
\end{figure}

\begin{figure}[!htbp]
\centering
\includegraphics[width=2.5in]{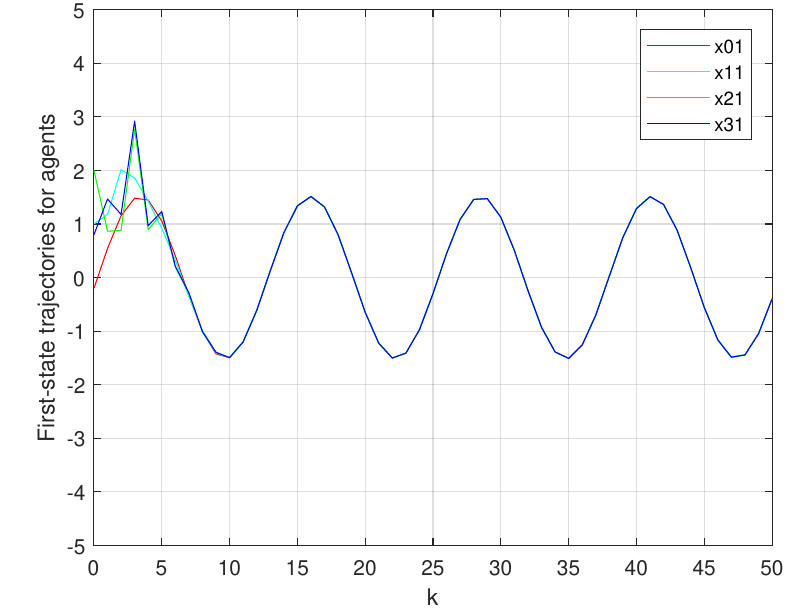}
\caption{The first state trajectories of each agent $x_{i}(k),i=0,1,2,3$.}
\label{fig:1-3}
\end{figure}

\begin{figure}[!htbp]
\centering
\includegraphics[width=2.5in]{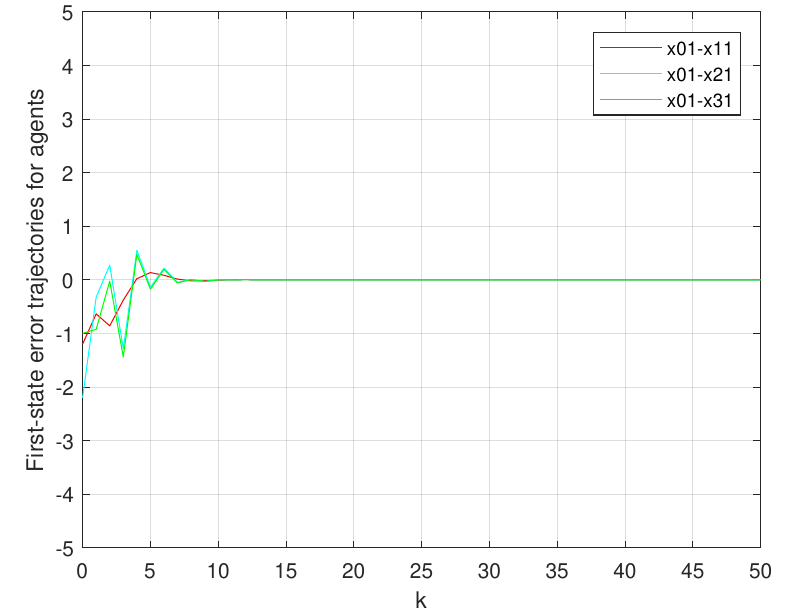}
\caption{The first state error trajectories by the proposed method.}
\label{fig:1-4}
\end{figure}

\begin{figure}[!htbp]
\centering
\includegraphics[width=2.5in]{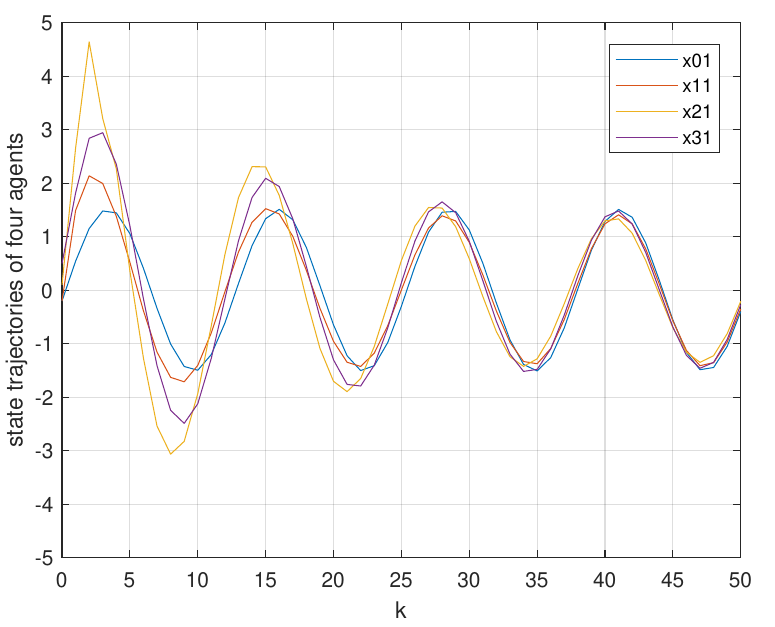}
\caption{The first state trajectories of each agent by the existing method.}
\label{fig:1-5}
\end{figure}

\begin{figure}[!htbp]
\centering
\includegraphics[width=2.5in]{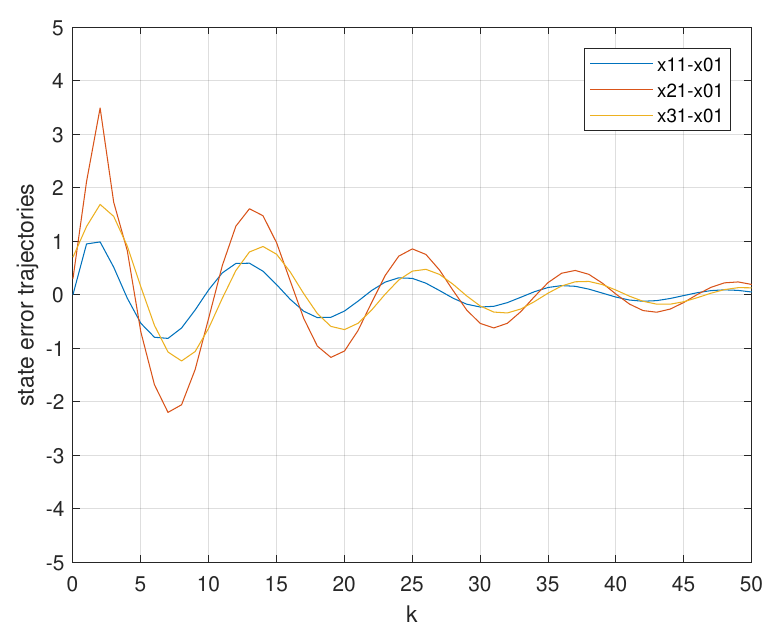}
\caption{The first state error trajectories by the existing method.}
\label{fig:1-6}
\end{figure}

%
%

\end{example}

\section{Conclusion}\label{sec:conclusion}
In this paper, we have studied the consensus problem for heterogeneous multi-agent systems by  a new feedforward control approach and optimal control theory.
Different from the traditional algorithms,
the proposed new distributed controller contains two parts:
a distributed feedforward controller was designed based on the information of agents and its neighbors,
and then an optimal consensus controller based on observers involving agent's historical state information was obtained by solving Riccati equations.
Different from the existing consensus control algorithms,
the current proposed algorithm has removed the requirements  for designing an additional distributed observer
and  for solving  a set of regulator equations.
It is shown that the corresponding cost function under the proposed distributed controllers is asymptotically optimal.
The proposed consensus algorithm can be directly applied to the  homogeneous systems case,
therefore, a unified design framework is provided to solve the consensus problem for heterogeneous system and homogeneous system.
Simulation example indicated the effectiveness of the proposed scheme, and  the proposed approach can achieve a faster convergence than the conventional algorithms.


\ifCLASSOPTIONcaptionsoff
  \newpage
\fi



%
\bibliographystyle{IEEEtran}
\bibliography{mybib}

\end{document}